\documentclass{amsart}
\usepackage{url}
\usepackage[utf8]{inputenc}
\usepackage[english]{babel}
\usepackage[usenames, dvipsnames]{color}
\usepackage[left=1.3cm,right=1.3cm,top=3cm,bottom=3cm]{geometry}
\usepackage{stackrel}

\usepackage[all,cmtip]{xy}

\usepackage{longtable}

\counterwithin{equation}{section}

\newtheorem{theorem}{Theorem}[section]
\newtheorem{definition}[theorem]{Definition}
\newtheorem{lemma}[theorem]{Lemma}

\theoremstyle{remark}
\newtheorem{remark}[theorem]{Remark}

\newcommand{\NN}{\mathbb{N}}

\newcommand{\KK}{\mathbb{K}}
\newcommand{\M}{\mathrm{M}}

\newcommand{\cF}{\mathcal{F}}
\newcommand{\cM}{\mathcal{M}}
\newcommand{\cU}{\mathcal{U}}

\DeclareMathOperator{\Fun}{Fun}
\DeclareMathOperator{\Cat}{Cat}
\DeclareMathOperator{\Lex}{Lex}
\DeclareMathOperator{\ex}{ex}

\DeclareMathOperator{\colim}{colim}
\DeclareMathOperator{\Res}{Res}
\DeclareMathOperator{\Spt}{\text{Spt}}
\DeclareMathOperator{\Spc}{\text{Spc}}
\DeclareMathOperator{\Grp}{\text{Grp}}

\newcommand{\Orb}{\mathrm{Orb}}
\DeclareMathOperator{\Stab}{stab}
\DeclareMathOperator{\incl}{inc}
\DeclareMathOperator{\loc}{loc}

\DeclareMathOperator{\Map}{Map}

\allowdisplaybreaks

\usepackage{lipsum}

\newcommand\blfootnote[1]{%
  \begingroup
  \renewcommand\thefootnote{}\footnote{#1}%
  \addtocounter{footnote}{-1}%
  \endgroup
}

\begin{document}
\title[An improvement of the Farrell-Jones conjecture for localising invariants]{An improvement of the Farrell-Jones conjecture for localising invariants}
\author{José Francisco Reis}
\address{Center for Mathematics and Applications (NOVA Math) and Department of Mathematics, NOVA SST}
\email{jfd.reis@campus.fct.unl.pt}
\thanks{}
\maketitle
\begin{abstract}
    The Farrell-Jones conjecture for lax monoidal finitary localising invariants was recently proved by Bunke-Kasprowski-Winges. In this short note, making use of the theory of noncommutative motives, we prove that the lax monoidal assumption is not necessary.
\end{abstract}
\pagenumbering{arabic}

\,

\section{Introduction}

\blfootnote{2020 Mathematics Subject Classification: 18F25, 19D10, 19D23, 19D55, 19E08.}

The computation of the algebraic $K$-theory of a group ring $RG$, for $R$ a commutative ring and $G$ a group (that can be infinite), is a very difficult problem.
The Farrell-Jones conjecture is a local-to-global statement that simplifies this problem. Roughly speaking, it claims that the algebraic $K$-theory of $RG$ is completely determined by the algebraic $K$-theory of the group rings $RV$, with $V$ a virtually cyclic subgroup\footnote{ Recall that a subgroup $V$ of $G$ is called virtually cyclic if it contains a cyclic subgroup of finite index.} of $G$.
This important conjecture was first formulated by Farrell and Jones in the 90's (see \cite{FJ93}) and is nowadays known to be true in several cases: 
hyperbolic groups,
finite dimensional CAT(0)-groups,
virtually solvable groups,
mapping class groups, etc.
For striking applications of the Farrell-Jones conjecture, we invite the reader to consult the surveys \cite{Luck11, Luck20, LR05}.

In the late 90's Davis and L\"uck \cite{DL98} proposed a general
setting for stating the Farrell-Jones conjecture:
let $\cF$ be a family of subgroups of $G$,
$G\Orb$ the orbit category\footnote{ Recall that the objects of $G\Orb$ are the left $G$-sets $G/V$, with $V$ a subgroup of $G$, and the morphisms are the maps of $G$-sets $G/V \rightarrow G/V'$.} of $G$,
$G_{\cF}\Orb$ the full subcategory of $G\Orb$ of those left $G$-sets with stabilizers in $\cF$, and $\mathbf{E}$ a functor from $G\Orb$ to the category $\Spt$ of Spectra.
Under these notations, the \emph{$\mathbf{E}$-assembly map} is the naturally induced map from the homotopy colimit of $\mathbf{E}$ over $G_{\cF}\Orb$ to the homotopy colimit of $\mathbf{E}$ over $G\Orb$, and the \emph{$\mathbf{E}$-isomorphism conjecture} is the claim that the $\mathbf{E}$-assembly map is a weak equivalence.
In \cite{DL98}, Davis and L\"uck proved that the classical Farrell-Jones isomorphism conjecture is equivalent to the $\mathbf{E}$-isomorphism conjecture when $\cF$ is the family of virtually cyclic  subgroups and $\mathbf{E}$ is the following composition
\begin{align}
\label{eq: davis and luck}
\mathbf{E}: G\Orb \stackrel{\overline{?}}{\longrightarrow} \Grp \stackrel{R[-]}{\longrightarrow} \text{Cat}_R \stackrel{\KK}{\longrightarrow} \Spt,
\end{align}
where $\Grp$ stands for the category of groupoids, $\text{Cat}_R$ for the category of $R$-linear categories, $\overline{?}$ is the  functor that sends a left $G$-set $G/V$ to the associated groupoid
$\overline{G/V}$, $R[-]$ stands for the $R$-linearization functor, and $\KK$ is the non-connective $K$-theory functor.
One of the advantages of this general setting is that by simply replacing $\KK$ by any other functor $H$, one immediately obtains a variant of the classical Farrell-Jones conjecture\footnote{ Given all these variations of the Farrell-Jones conjecture, Balmer and Tabuada, making use of the theory of noncommutative motives, explicitly described in \cite{BT13} the Fundamental Isomorphism Conjecture which implies all these variants of the Farrell-Jones conjecture.}.
Examples of such functors $H$ include  homotopy $K$-theory \cite{BL06}, Hochschild and Cyclic Homology \cite{LR06}, topological Hochschild Homology \cite{Luck10}, etc.

In the late 00's, Bartels and Reich \cite{BR07} introduced the Farrell-Jones conjecture with coefficients in a $R$-linear category $C$ equipped with a right $G$-action.
Similarly to the setting of Davis and L\"uck, they considered the following composition
\begin{align}
\label{eq: bartels and reich}
\mathbf{E}: G\Orb \stackrel{C *_{G} - }{\longrightarrow} \text{Cat}_R \stackrel{\KK}{\longrightarrow} \Spt,
\end{align}
where $C *_{G} -$ is the functor that sends a left $G$-set $G/V$ to its ``tensor product over $G$ with $C$", consult \cite[\S 2]{BR07} for details.
In the particular case where $C$ is the $R$-linear category $R_{\oplus}$ obtained from $R$ by formally adding finite direct sums, equipped with the trivial $G$-action, the above composition (\ref{eq: bartels and reich}) reduces to (\ref{eq: davis and luck}).
Hence, the setting of Bartels and Reich generalizes the one of Davis and L\"uck.
Other choices of $C$ lead, for example, to the Farrell-Jones conjecture with coefficients in a twisted group ring or, more generally, with coefficients in a crossed product ring.

Recently, Bunke, Kasprowski and Winges in \cite{BKW21} generalized the setting of Bartels and Reich to the realm of $\infty$-categories.
Since we will work in this setting, we now explain it in detail.
In what follows, we will denote by $\Cat^{\Lex}_{\infty,*}$ the (large) $\infty$-category of small left-exact $\infty$-categories\footnote{ Recall that a (small) $\infty$-category is called \emph{left-exact} if it is point
ed and admits all finite limits.} and finite limit preserving functors.
Let $BG$ be the groupoid with one object and group of automorphisms $G$. Note that there is a canonical inclusion $j^G:BG \rightarrow G\Orb$ sending the unique object of $BG$ to the left $G$-set $G/\{1\}$.
Also, let $C$ be a small left-exact $\infty$-category equipped with a right $G$-action, i.e., an object of $\Fun(BG, \Cat^{\Lex}_{\infty,*})$.
In what follows, we will denote the left Kan-extension of $C$ along $j^G$ by $C_G:G\Orb \rightarrow \Cat^{\Lex}_{\infty,*}$.
Finally, let $H: \Cat^{\Lex}_{\infty,*}\rightarrow  \M$ be a functor with values in a cocomplete stable $\infty$-category.
Under these notations, Bunke, Kasprowski and Winges considered the following composition:
\begin{align}
    \label{eq: BKW}
    \mathbf{E}: G\Orb \stackrel{C_G}{\longrightarrow}  \Cat^{\Lex}_{\infty,*} \stackrel{H}{\longrightarrow}  \M.
\end{align}
As explained in \cite[Example 1.6]{BCKW19}, in the particular case when the $\infty$-category $C$ is $R$-linear, the functor $C_G$ reduces to $C*_G - $.
Hence, the setting of Bunke, Kasprowski, and Winges generalizes the setting of Bartels and Reich.
Note that the setting of Bunke, Kasprowski and Winges is very general: one can vary the group $G$, the family of subgroups $\cF$, the $\infty$-category of coefficients $C$ and also the functor $H$.
Nevertheless, they were able to prove that the $\mathbf{E}$-isomorphism conjecture (with respect to (\ref{eq: BKW})) holds for a large class of lax monoidal functors $H$; consult Theorem \ref{main theorem 1} below.
The goal of this short note is to prove that the lax monoidal assumption on $H$ is not needed; consult Theorem \ref{main theorem 2} below.

\section{Statement of results}
\label{sec: defs and results}

We start by recalling from \cite{BKW21} some key definitions. 

\begin{definition}
\label{def: finitary localising}
A functor $H: \Cat^{\Lex}_{\infty,*}\rightarrow  \M$, with values in a cocomplete stable $\infty$-category, is called a \emph{finitary localising invariant} if it preserves zero objects and filtered colimits, sends excisive squares to pushout squares and inverts Morita equivalences.
\end{definition}

Let $\Spc$ denote the $\infty$-category of spaces. Recall from \cite[\S 2.1]{BKW21} that an object $K$ of a cocomplete $\infty$-category $\M$ is called \emph{compact} if the functor $\Map_{\M}(K,-):\M \rightarrow \Spc$ preserves filtered colimits.

\begin{definition}[{\cite[Definitions 2.4 and 2.5]{BKW21}}]
\label{def: phantom object-equivalence}
    Let $\M$ a cocomplete $\infty$-category.
    \renewcommand{\labelenumi}{(\roman{enumi})}
    \begin{enumerate}
        \item An object $M$ of $\M$ is called a \emph{phantom object} if $\Map_\M(K, M) \simeq *$ for every compact object $K$ of $\M$.
        \item A morphism $m: M \rightarrow M'$ in $\M$ is called a \emph{phantom equivalence} if $\Map_\M(K, m)$ is an equivalence of spaces for every compact object $K$ of $\M$.
    \end{enumerate}
\end{definition}

\begin{remark}
\label{remark: phathom=equivalence}
    When $\M$ is compactly generated, the notions of phantom equivalence and equivalence agree, see \cite[Rk. 2.6]{BKW21}.
\end{remark}

The main result of \cite{BKW21} is the following:

\begin{theorem}[{\cite[Theorem 1.4]{BKW21}}]
\label{main theorem 1}
Let $G$ be a group, $\cF$ a family  of subgroups of $G$, $C$ a small left-exact $\infty$-category equipped with a right  $G$-action, and $H: \Cat^{\Lex}_{\infty,*}\rightarrow  \M$ a functor with values in a cocomplete stable $\infty$-category.
Assume that $G$ is a Dress-Farrell-Hsiang-Jones (DFHJ) group relative to $\cF$ (consult \cite[Definition 7.1]{BKW21}) and that $H$ is a lax monoidal finitary localising invariant.
Under these assumptions, the assembly map
\begin{align}
    \label{eq: assembly-0}
    A_{\cF,H \circ C_G} : \colim_{G_{\cF}\Orb}(H \circ C_G) \longrightarrow H(\colim_{BG} C)
\end{align}
is a phantom equivalence.
\end{theorem}

\begin{remark}
Note that $G/G$ is the terminal object of the orbit category $G\Orb$. Hence, the colimit of the composition $H\circ C_G$ over $G\Orb$ is given by $(H \circ C_G)(G/G)$.
We have the computation $\colim_{BG}C \simeq \colim_{BG}\Res^{G}_G(C) \simeq C_G(G/G)$, see \cite[(1.3)]{BKW21}. Consequently, the assembly map (\ref{eq: assembly-0}) agrees with the $\mathbf{E}$-assembly map of the above composition (\ref{eq: BKW}).
\end{remark}

Theorem \ref{main theorem 1} (together with Remark \ref{remark: phathom=equivalence}) proves the Farrell-Jones conjecture for a very large class of groups $G$, of coefficients $C$, and of functors $H$.
In this short note, making use of the theory of noncommutative motives, we improve Theorem \ref{main theorem 1} as follows:

\begin{theorem}
\label{main theorem 2}
Theorem \ref{main theorem 1} holds ipsis verbis
with $H$ a finitary localising invariant (the lax monoidal assumption is not necessary).
\end{theorem}

\begin{remark}
\renewcommand{\labelenumi}{(\roman{enumi})}
\begin{enumerate}
    \item As explained in Section \ref{section: the proof} below, the proof of Theorem \ref{main theorem 2} makes essential use of Theorem \ref{main theorem 1}.
\item Theorem \ref{main theorem 2} was incorporated in the latest version of \cite{BKW21}; consult \cite[Rk. 1.5]{BKW21}.
\end{enumerate}
\end{remark}

\begin{remark}[Lax monoidal]
    Note that being lax monoidal is an extra structure on a functor and not a property of a functor\footnote{Let $C$ and $D$ be two symmetric monoidal $\infty$-categories such that $D$ admits all colimits. As proved in \cite[Proposition 2.12]{Glasman16}, to give a lax monoidal structure on a functor $H: C \rightarrow D$ is equivalent to give a $E^\infty$-monoid structure on the object $H$ of the symmetric monoidal $\infty$-category $\Fun(C,D)$ (equipped with the Day convolution product).}.
    Hence, Theorem \ref{main theorem 2} shows that a finitary localizing invariant $H$ does not need to be equipped with this extra structure in order for the assembly map (\ref{eq: assembly-0}) to be a phantom equivalence.
\end{remark}
\begin{remark}[Full Farrell-Jones conjecture]
   As explained in \cite[Definition 1.6]{BKW21}, there is a variant of the Farrell-Jones conjecture called the \emph{Full Farrell-Jones conjecture}.
   Let $H: \Cat^{\Lex}_{\infty,*}\rightarrow  \M$ be a lax monoidal finitary localising invariant.
   In addition to the above Theorem \ref{main theorem 1}, Bunke, Kasprowski and Winges also proved in \cite{BKW21} that the class of groups which satisfy the Full Farrell-Jones conjecture for $H$ is quite large and closed under several constructions; consult \cite[Theorem 1.7]{BKW21}.
   Similarly to Theorem \ref{main theorem 1}, this result also holds ipsis verbis with $H$ a finitary localizing invariant (the lax monoidal assumption is not necessary).
   Simply follow the same proof of \cite[Theorem 1.7]{BKW21} and replace the references to      \cite[Theorems 5.1, 6.1 and 7.1]{BKW21} by the reference to Theorem \ref{main theorem 2} above. 
\end{remark}

\section{Proof of Theorem \ref{main theorem 2}}
\label{section: the proof}
Let $\Cat^{\ex}_{\infty,*}$ be the (large) $\infty$-category of small stable $\infty$-categories, $\Stab: \Cat^{\Lex}_{\infty,*} \rightarrow \Cat^{\ex}_{\infty,*}$ the stablization functor and $\incl: \Cat^{\ex}_{\infty,*} \rightarrow \Cat^{\Lex}_{\infty,*}$ the inclusion functor; consult \cite[\S 7.4]{BCKW19}.
As proved in \cite[Lemma 7.41]{BCKW19}, we have the following adjunction:

\begin{align*}
\xymatrix@C=4.5pc @R=2pc{
\Cat^{\Lex}_{\infty,*} \ar@<-.45pc>[d]_{\Stab}\\
\Cat^{\ex}_{\infty,*} \ar@<-.45pc>[u]_{\incl}.}
\end{align*}

Moreover, as explained in \cite[\S 6.2]{BCKW19}, the composition $H \circ \incl$ is a localising invariant in the sense of \cite{BGT13}.
Furthermore, as proved in \cite[Lemma 6.9(1)]{BCKW19}, the functors $H$ and $H \circ \incl \circ \Stab$ are naturally equivalent.
Note that in \cite{BCKW19,BKW21} a localising invariant is referred to as a stable finitary localising invariant; in what follows, we stick to this latter terminology.

Recall from \cite[\S 8]{BGT13} the construction of the universal stable finitary localising invariant $\cU_{\loc}: \Cat^{\ex}_{\infty, *} \rightarrow \cM_{\loc}$; due to this universal property, $\cM_{\loc}$ is called the (large) $\infty$-category of noncommutative motives; consult \cite{Tab15}.
Since the functor $H \circ \incl$ is a stable finitary localising invariant, \cite[Theorem 8.7]{BGT13} implies that it factors through $\cU_{\loc}$; let us denote by $\overline{H \circ \incl}$ the induced functor.
Under these notations, we have the following commutative diagram (up to equivalence):
\begin{align}
\label{the diagram}
\begin{split}
\xymatrix@C=4.5pc @R=2pc{
G\Orb \ar[rd]^-{C_G} \ar@/^1.5pc/[rrd]^-{H \circ C_G} & & \\
BG \ar[u]^-{j^G}  \ar[r]_-{C} & \Cat^{\Lex}_{\infty,*} \ar[r]^-{H} \ar@<-.4pc>[d]_{\Stab}  & \mathrm{M}\\
 & \Cat^{\ex}_{\infty,*} \ar@<-.4pc>[u]_{\incl} \ar[d]_-{\cU_{\loc}} \ar[ru]_{H \circ \incl} &\\
& \cM_{\loc} \ar@/_1.5pc/[ruu]_-{\overline{H \circ \incl}}  &.
}
\end{split}
\end{align}

Since $\cU_{\loc}$ is a stable finitary localising invariant, \cite[Lemma 6.9(2)]{BCKW19} implies that the composition $\cU_{\loc} \circ \Stab$ is a finitary localising.
Moreover, as explained in \cite[Theorem 5.8]{BGT14} and \cite[\S 1]{BKW21}, respectively, the functors $\cU_{\loc}$ and $\Stab$ are symmetric monoidal.
This implies that the composition $\cU_{\loc} \circ \Stab$ is a lax monoidal finitary localising invariant.
Consequently,  Theorem \ref{main theorem 1} (with $H$ replaced by $\cU_{\loc} \circ \Stab$) implies that the assembly map
\begin{align}
    \label{eq: assembly-1}
    A_{\cF, (\cU_{\loc}\circ \Stab \circ C_G)}: \colim_{G_{\cF}\Orb}  (\cU_{\loc}\circ \Stab \circ C_G) \longrightarrow (\cU_{\loc}\circ \Stab) (\colim_{BG}C)
\end{align}
is a phantom equivalence.
Let us denote by $\mathrm{cof}$ the  cofiber of (\ref{eq: assembly-1}). As explained in the proofs of \cite[Theorem 1.4 and Proposition 2.33]{BKW21}, the diagonal map $\Delta: \mathrm{cof} \rightarrow \Pi_\NN \mathrm{cof}$ factors through the canonical map $\oplus_{\NN}\mathrm{cof} \rightarrow \Pi_\NN \mathrm{cof}$. 
Therefore, since the functor $\overline{H \circ \incl}$ is colimit preserving, it follows from Lemma \ref{lemma: a functor preserves a nice property} below that $\Delta: (\overline{H \circ \incl})(\mathrm{cof}) \rightarrow \Pi_{\NN} (\overline{H \circ \incl})(\mathrm{cof})$
factors through the canonical morphism $\oplus_{\NN}(\overline{H \circ \incl})(\mathrm{cof})\rightarrow \Pi_{\NN}(\overline{H \circ \incl})(\mathrm{cof})$.
Following \cite[Lemma 2.8]{BKW21}, this implies, in particular, that $(\overline{H \circ \incl})(\mathrm{cof})$ is a phantom object.

Now, let us apply the functor $\overline{H \circ \incl}$ to the above assembly map (\ref{eq: assembly-1}).
Since the functor $\overline{H \circ \incl}$ is colimit preserving, it follows from the above commutative diagram (\ref{the diagram}) that $(\overline{H \circ \incl})(\ref{eq: assembly-1})$ identifies with the assembly map (\ref{eq: assembly-0}).
Moreover, $(\overline{H \circ \incl})(\mathrm{cof})$ identifies with the cofiber  of (\ref{eq: assembly-0}).
Consequently, since $(\overline{H \circ \incl})(\mathrm{cof})$ is a phantom object, we conclude that the assembly map (\ref{eq: assembly-0}) is a phantom equivalence; consult \cite[Rk. 2.6]{BKW21}. \hfill $\qed$

\begin{lemma}
\label{lemma: a functor preserves a nice property}
Let $o$ be an object of an $\infty$-category $\mathcal{E}$ such that the diagonal map $\Delta: o \rightarrow \Pi_\NN o$ factors through the canonical map $\oplus_{\NN} o \rightarrow \Pi_\NN o$.
Given a colimit preserving functor $F : \mathcal{E}\rightarrow \mathcal{D}$, the diagonal map $\Delta: F(o)\rightarrow \Pi_\NN F(o)$ factors through the canonical map $\oplus_{\NN}F(o) \rightarrow \Pi_\NN F(o)$. 
\end{lemma}
\begin{proof}
Apply the functor $F$ to the factorization of the diagonal map $\Delta :o \rightarrow \Pi_\NN o$ through the canonical map $\oplus_{\NN} o \rightarrow \Pi_\NN o$.
Thanks to the universal property of the product and to the fact that $F$ is colimit preserving, one hence obtains a factorization of the diagonal map $\Delta: F(o)\rightarrow \Pi_\NN F(o)$ through the canonical map $\oplus_{\NN}F(o) \rightarrow \Pi_\NN F(o)$.
\end{proof}

\noindent \textbf{Acknowledgments}. I would like to thank Gonçalo Tabuada for many discussions regarding \cite{BKW21} and the anonymous referee for suggestions concerning the improvement of this paper.
This work is funded by national funds through the FCT - Fundação para a Ciência e a Tecnologia, I.P., under the scope of the projects UIDB/00297/2020 and UIDP/00297/2020 (Center for Mathematics and Applications) and the PhD scholarship SFRH/BD/144305/2019. 

\bibliographystyle{siam}
\bibliography{biblio}

\begin{thebibliography}{10}

\bibitem{BT13}
{\sc P.~Balmer and G.~Tabuada}, {\em Fundamental isomorphism conjecture via
  non-commutative motives}, Math. Nachr., \textbf{286} (2013), pp.~791--798.

\bibitem{BL06}
{\sc A.~Bartels and W.~L{\"u}ck}, {\em Isomorphism conjecture for homotopy
  {{\(K\)}}-theory and groups acting on trees}, J. Pure Appl. Algebra,
  \textbf{205} (2006), pp.~660--696.

\bibitem{BR07}
{\sc A.~Bartels and H.~Reich}, {\em Coefficients for the {Farrell}-{Jones}
  conjecture}, Adv. Math., \textbf{209} (2007), pp.~337--362.

\bibitem{BGT13}
{\sc A.~J. Blumberg, D.~Gepner, and G.~Tabuada}, {\em A universal
  characterization of higher algebraic {{\(K\)}}-theory}, Geom. Topol.,
  \textbf{17} (2013), pp.~733--838.

\bibitem{BGT14}
\leavevmode\vrule height 2pt depth -1.6pt width 23pt, {\em Uniqueness of the
  multiplicative cyclotomic trace}, Adv. Math., \textbf{260} (2014),
  pp.~191--232.

\bibitem{BCKW19}
{\sc U.~Bunke, D.-C. Cisinski, D.~Kasprowski, and C.~Winges}, {\em {Controlled
  objects in left-exact $\infty$-categories and the Novikov conjecture}},
  (2019).
\newblock Available at https://arxiv.org/abs/1911.02338.

\bibitem{BKW21}
{\sc U.~Bunke, D.~Kasprowski, and C.~Winges}, {\em {On the Farrell-Jones
  conjecture for localising invariants}},  (2021).
\newblock Available at https://arxiv.org/abs/2111.02490.

\bibitem{DL98}
{\sc J.~F. Davis and W.~L{\"u}ck}, {\em Spaces over a category and assembly
  maps in isomorphism conjectures in {{\(K\)}}- and {{\(L\)}}-theory},
  \(K\)-Theory, \textbf{15} (1998), pp.~201--252.

\bibitem{FJ93}
{\sc {F. T. Farrell and L. E. Jones}}, {\em Isomorphism conjectures in
  algebraic {{\(K\)}}-theory}, J. Am. Math. Soc., \textbf{6} (1993),
  pp.~249--297.

\bibitem{Glasman16}
{\sc S.~Glasman}, {\em Day convolution for {{\(\infty\)}}-categories}, Math.
  Res. Lett., \textbf{23} (2016), pp.~1369--1385.

\bibitem{Luck10}
{\sc W.~L{\"u}ck}, {\em On the {Farrell}-{Jones} and related conjectures}, in
  Cohomology of groups and algebraic \(K\)-theory. Selected papers of the
  international summer school on cohomology of groups and algebraic
  \(K\)-theory, Hangzhou, China, July 1--3, 2007, Somerville, MA: International
  Press; Beijing: Higher Education Press, 2010, pp.~269--341.

\bibitem{Luck11}
\leavevmode\vrule height 2pt depth -1.6pt width 23pt, {\em {{\(K\)}}- and
  {{\(L\)}}-theory of group rings}, in Proceedings of the international
  congress of mathematicians (ICM 2010), Hyderabad, India, August 19--27, 2010.
  Vol. II: Invited lectures, Hackensack, NJ: World Scientific; New Delhi:
  Hindustan Book Agency, 2011, pp.~1071--1098.

\bibitem{Luck20}
\leavevmode\vrule height 2pt depth -1.6pt width 23pt, {\em Assembly maps}, in
  Handbook of homotopy theory, Boca Raton, FL: CRC Press, 2020, pp.~851--890.

\bibitem{LR05}
{\sc W.~L{\"u}ck and H.~Reich}, {\em The {Baum}-{Connes} and the
  {Farrell}-{Jones} conjectures in {{\(K\)}}- and {{\(L\)}}-theory}, in
  Handbook of \(K\)-theory. Vol. 1 and 2, Berlin: Springer, 2005, pp.~703--842.

\bibitem{LR06}
\leavevmode\vrule height 2pt depth -1.6pt width 23pt, {\em Detecting
  {{\(K\)}}-theory by cyclic homology}, Proc. Lond. Math. Soc. (3), \textbf{93}
  (2006), pp.~593--634.

\bibitem{Tab15}
{\sc G.~{Tabuada}}, {\em Noncommutative motives \emph{(with a preface by Yuri
  I. Manin)}}, vol.~\textbf{63} of University Lecture Series, Providence, RI:
  American Mathematical Society (AMS), 2015.

\end{thebibliography}

\end{document}